\documentclass[11pt]{amsart}

\usepackage{amsfonts, amsmath, amscd}
\usepackage[psamsfonts]{amssymb}
\usepackage{soul}
\usepackage{graphics, picture}

\usepackage{amssymb}

\usepackage[usenames]{color}

\newtheorem{theorem}{Theorem}[section]
\newtheorem{lemma}[theorem]{Lemma}
\newtheorem{corollary}[theorem]{Corollary}
\newtheorem{question}[theorem]{Question}

\newtheorem{proposition}[theorem]{Proposition}

\newtheorem{example}[theorem]{Example}

\newtheorem{claim}[theorem]{Claim}

\numberwithin{equation}{section}

\newcommand{\CC}{C_k}
\newcommand{\NN}{\omega}

\newcommand{\UU}{\mathcal{U}}

\newcommand{\w}{\omega}

\newcommand{\V}{\mathcal{V}}

\newcommand{\KK}{\mathcal{K}}

\newcommand{\IR}{\mathbb{R}}
\newcommand{\II}{\mathbb{I}}

\newcommand{\e}{\varepsilon}

\renewcommand{\phi}{\varphi}

\newcommand{\U}{\mathcal U}
\newcommand{\W}{\mathcal W}

\newcommand{\supp}{\mathrm{supp}}

\newcommand{\Op}{\mathcal{O}\/}
\newcommand{\uhr}{\upharpoonright}

\input xy
\xyoption{all}

\title[The Ascoli property for function spaces]{The Ascoli property for function spaces}

\author[S. Gabriyelyan]{Saak Gabriyelyan}
\address{Department of Mathematics, Ben-Gurion University of the Negev, Beer-Sheva, P.O. 653, Israel}
\email{saak@math.bgu.ac.il}

\author[J. Greb\'{\i}k]{Jan Greb\'{\i}k}
\address{Institute of Mathematics, Czech Academy of Sciences, Czech Republic}
\email{Greboshrabos@seznam.cz}

\author[J. K\c{a}kol]{Jerzy K\c{a}kol}
\address{A. Mickiewicz University $61-614$ Pozna{\'n}, Poland and Institute of Mathematics, Czech Academy of Sciences, Czech Republic}
\email{kakol@amu.edu.pl}

\author[L. Zdomskyy]{Lyubomyr Zdomskyy}
\address{Kurt G\"odel Research Center for Mathematical Logic, University of Vienna, W\"ahringer Stra\ss e 25, A-1090 Wien, Austria.}
\email{lzdomsky@gmail.com}



\thanks{The second author was supported by the GACR project 15-34700L and RVO: 67985840. The third  author was supported by Generalitat Valenciana, Conselleria d'Educaci\'{o}, Cultura i Esport, Spain, Grant PROMETEO/2013/058 and by  the GA\v{C}R project 16-34860L and RVO: 67985840, and gratefully acknowledges also the financial support he received from the Kurt Goedel Research Center in Wien for his research visit in days 15.04-24.04 2016. The fourth author  would like to thank  the Austrian Science Fund FWF (Grant I 1209-N25) for generous support for this research. The collaboration of the second and the fourth authors was partially supported by the Czech Ministry of Education grant 7AMB15AT035 and RVO: 67985840.}

\begin{document}

\begin{abstract}

The paper deals with Ascoli spaces $C_{p}(X)$ and $\CC(X)$ over
Tychonoff
spaces $X$. The class of Ascoli spaces $X$, i.e. spaces $X$ for which any compact
subset $\KK$ of $\CC(X)$ is evenly continuous, essentially includes the class of
$k_\IR$-spaces. First we prove that if  $C_p(X)$ is Ascoli,  then it is
$\kappa$-Fr\'echet--Urysohn. If $X$ is cosmic, then $C_p(X)$ is Ascoli iff it
 is $\kappa$-Fr\'echet--Urysohn.  This leads to the following
extension of a  result
of Morishita:   If for  a \v{C}ech-complete space $X$ the space $C_p(X)$ is Ascoli,
then $X$ is scattered. If $X$ is scattered and stratifiable, then $C_p(X)$ is an Ascoli
 space. Consequently: (a) If $X$ is a complete metrizable space,
then $C_p(X)$ is Ascoli iff $X$ is scattered. (b) If $X$ is  a \v{C}ech-complete Lindel\"of space,
then $C_p(X)$ is Ascoli iff $X$ is scattered iff $C_{p}(X)$ is Fr\'echet-Urysohn. Moreover, we prove that for  a paracompact space $X$ of point-countable type the following conditions are equivalent: (i) $X$ is locally compact. (ii) $\CC(X)$ is a  $k_\IR$-space. (iii) $\CC(X)$ is an Ascoli space. The Asoli spaces     $\CC(X,\II)$ are also studied.
\end{abstract}

\maketitle

%%%%%%%%%%%%%%%%%%%%%%%%%%%
%%%%%%%%%%%%%%%%%%%%%%%%%%%
%%%%%%%%%%%%%%%%%%%%%%%%%%%
%%%%%%%%%%%%%%%%%%%%%%%%%%%

%%%%%%%%%%%%%%%%%%%%%%%%%%%
%%%%%%%%%%%%%%%%%%%%%%%%%%%

\section{Introduction }

%\begin{theorem} \label{t:Ck-metr-Ascoli}
%Let $X$ be a metrizable space. Then $\CC(X)$ is Ascoli if and only if $X$ is locally compact.
%\end{theorem}

%%%%%%%%%%%%%%%%%%%%%%%%%%%
%%%%%%%%%%%%%%%%%%%%%%%%%%%
%%%%%%%%%%%%%%%%%%%%%%%%%%%
%%%%%%%%%%%%%%%%%%%%%%%%%%%

Various topological properties generalizing metrizability have been  intensively
studied  both by topologists and analysts  for  a long time, and the following
diagram gathers some of the most important concepts:
\[
\xymatrix{
& \mbox{$\kappa$-Fr\'{e}chet--Urysohn} & & & & \\
\mbox{metric} \ar@{=>}[r] & {\mbox{Fr\'{e}chet--}\atop\mbox{Urysohn}} \ar@{=>}[r] \ar@{=>}[u] & \mbox{sequential} \ar@{=>}[r] &  \mbox{$k$-space} \ar@{=>}[r] &  \mbox{$k_\IR$-space} \ar@{=>}[r] &  {\mbox{Ascoli}\atop\mbox{space}}. }
\]
Note that none of these implications is reversible.
The study of the above concepts for the function spaces with various topologies has a rich history and is also nowadays an active area of research, see  \cite{Arhangel,kak,mcoy,Tkachuk-Book-1} and references therein.

For Tychonoff  topological spaces $X$ and $Y$, we denote by $\CC(X,Y)$ and $C_p(X,Y)$ the space $C(X,Y)$ of all continuous  functions from $X$ into $Y$ endowed with the compact-open topology or the pointwise topology, respectively. If $Y=\IR$, we shall write $\CC(X)$ and $C_p(X)$, respectively.

It is well-known that $C_p(X)$ is metrizable if and only if $X$ is countable. Pytkeev, Gerlitz and Nagy (see \S 3 of \cite{Arhangel}) characterized  spaces $X$ for which $C_p(X)$ is Fr\'{e}chet--Urysohn, sequential or a $k$-space (these properties coincide for the spaces $C_p(X)$). Sakai in \cite{Sak2} described all   spaces $X$ for which $C_p(X)$ is $\kappa$-Fr\'{e}chet--Urysohn, see Theorem \ref{t:Sakai-Cp-k-FU} below.
 However, very little is  known about spaces $X$ for which $C_p(X)$ is an Ascoli space
or a $k_\IR$-space.

Following \cite{BG}, a   space $X$ is called an {\em Ascoli space} if each compact subset $\KK$ of $\CC(X)$ is evenly continuous, that is, the map $X\times\mathcal K\ni (x,f)\mapsto f(x)\in\mathbb R $ is continuous.  Equivalently, $X$ is Ascoli if the natural evaluation map $X\hookrightarrow \CC(\CC(X))$ is an embedding, see \cite{BG}. Recall that a  space $X$ is called a {\em $k_\IR$-space} if a real-valued function $f$ on $X$ is continuous if and only if its restriction $f|_K$ to any compact subset $K$ of $X$ is continuous. It is known that every $k_\IR$-space is Ascoli, but the converse is in general not true, see \cite{BG}.

The class of Ascoli spaces  was  introduced  in \cite{BG}. The question for which  spaces  $X$ the space $C_p(X)$ is Ascoli or a $k_\IR$-space is posed in \cite{GKP}. It turned out that for spaces of the form $C_p(X)$,  the Ascoli property  is formally stronger than the  $\kappa$-Fr\'echet--Urysohn one. This follows from the following
\begin{theorem} \label{t:Cp-Ascoli-k-FU}
\begin{itemize}
 \item[{\rm (i)}] If $C_p(X)$ is Ascoli,  then it is $\kappa$-Fr\'echet--Urysohn.
 \item[{\rm (ii)}]  If $C_p(X)$ is $\kappa$-Fr\'echet--Urysohn and every compact $K\subset \CC(C_p(X))$ is first-countable, then $C_p(X)$ is Ascoli.
\end{itemize}
\end{theorem}

Recall  that a regular space $X$ is {\em cosmic} if it is a continuous image of a separable metrizable space, see \cite{Mich}. Michael proved in \cite{Mich} that every compact subset of a cosmic space is metrizable, and if $X$ is a cosmic space then $C_p(X)$ and hence $C_p(C_p(X))$ are cosmic. So all compact subsets of $C_p(C_p(X))$ and hence $\CC(C_p(X))$ are metrizable. This remark and  Theorem \ref{t:Cp-Ascoli-k-FU} imply
\begin{corollary} \label{c:Cp-k-FU-metr}
If $X$ is a cosmic space, then $C_p(X)$ is Ascoli if and only if it is $\kappa$-Fr\'echet--Urysohn. % (if and only if $X$ has property $(\kappa)$).
\end{corollary}
%\begin{proof}
%If $X$ is cosmic, then $C_p(X)$ and hence $C_p(C_p(X))$ are cosmic spaces by \cite{Mich}. So all compact subsets of $C_p(C_p(X))$ and hence $\CC(C_p(X))$ are metrizable. Now we apply Theorem \ref{t:Cp-Ascoli-k-FU}.
%\end{proof}

The second principle result of Section \ref{sec:Cp} is the following theorem, which extends an unpublished result of Morishita \cite[Theorem 10.7]{husek} and \cite[Corollary 4.2]{cascales}, see also Corollary \ref{c:Cp-Ascoli-FU} below.
\begin{theorem}  \label{t:Cech-complete-Ascoli-scat}
\begin{itemize}
 \item[(i)] If $X$ is \v{C}ech-complete and $C_p(X)$ is Ascoli, then $X$ is scattered.
 \item[(ii)] If $X$ is scattered and stratifiable, then $C_p(X)$ is an Ascoli space.
\end{itemize}
\end{theorem}

Since a metrizable space $X$ is \v{C}ech-complete if and only if it is completely metrizable, and since every metrizable space is stratifiable,  Theorem \ref{t:Cech-complete-Ascoli-scat} implies
\begin{corollary} \label{c:Cp-Ascoli-metr-scat}
If $X$ is a completely metrizable (and separable) space, then $C_p(X)$ is Ascoli if and only if $X$ is scattered (and countable).
\end{corollary}

The following corollary strengthens also Proposition 6.6 of \cite{GKP}.
\begin{corollary} \label{c:compact-Ascoli-scat}
Let $X$ be a compact space. Then $C_p(X)$ is Ascoli if and only if $C_p(X)$ is Fr\'{e}chet--Urysohn if and only if $X$ is scattered.
\end{corollary}

The second part of our paper deals with the Ascoli spaces $\CC(X)$. In \cite{Pol-1974} Pol gave a complete characterization of those first-countable paracompact spaces $X$ for which the space $\CC(X,\II)$ is a $k$-space, where $\II=[0,1]$.

In \cite{Gabr-C2} the first named author described all zero-dimensional metric spaces $X$ for which the space $\CC(X,2)$ is Ascoli, where $2=\{ 0,1\}$ is the doubleton.

On the other hand, it is proved in \cite{GKP} that if $X$ is a first-countable paracompact $\sigma$-space, then $\CC(X,\II)$ is Ascoli if and only if $\CC(X)$ is Ascoli if and only if $X$ is a locally compact metrizable space. However this result does not cover the case for $X$ being a  non-metrizable compact space $X$ for which clearly the Banach space $\CC(X)$ is Ascoli.
The next theorem, which is the main result of     Section \ref{sec:Ck},
 extends all results mentioned above.
We prove the following
\begin{theorem} \label{t:Ck-Ascoli-point-count-type}
For a paracompact space $X$ of point-countable type the following conditions are equivalent:
\begin{enumerate}
\item[{\rm (i)}] $X$ is  locally compact;
\item[{\rm (ii)}] $X=\bigoplus_{i\in\kappa} X_i$, where all $X_i$ are Lindel\"{o}f locally compact spaces;
\item[{\rm (iii)}] $\CC(X)$ is a  $k_\IR$-space;
\item[{\rm (iv)}] $\CC(X)$ is an Ascoli space;
\item[{\rm (v)}] $\CC(X,\II)$  is a  $k_\IR$-space;
\item[{\rm (vi)}] $\CC(X,\II)$ is an Ascoli space.
\end{enumerate}
In cases (i)--(vi), the spaces $\CC(X)$  and $\CC(X,\II)$ are  homeomorphic to
 products of families of complete metrizable spaces.
\end{theorem}

In our forthcoming paper \cite{GGKZ-2} we show that the paracompactness assumption on $X$ cannot be omitted in Theorem \ref{t:Ck-Ascoli-point-count-type} and we provide the first  $C_{p}$-example of an Ascoli space not being a  $k_\IR$-space.

%In spite of all notions in the diagram above differ essentially in the class of Tychonoff spaces, the Pytkeev--Gerlitz--Nagy theorem (\cite[Theorem~II.3.7]{Arhangel}) shows that the properties to be  Fr\'{e}chet--Urysohn, sequential or a $k$-space coincide for the spaces $C_p(X)$.

%%%%%%%%%%%%%%%%%%%%%%%%%%%
%%%%%%%%%%%%%%%%%%%%%%%%%%%
%%%%%%%%%%%%%%%%%%%%%%%%%%%
%%%%%%%%%%%%%%%%%%%%%%%%%%%

\section{The Ascoli property for $C_p(X)$} \label{sec:Cp}

%%%%%%%%%%%%%%%%%%%%%%%%%%%
%%%%%%%%%%%%%%%%%%%%%%%%%%%
%%%%%%%%%%%%%%%%%%%%%%%%%%%
%%%%%%%%%%%%%%%%%%%%%%%%%%%

Let $X$ be a Tychonoff space and $h\in C(X)$. Then  the sets of the form
\[
[h,F,\e]:= \{ f\in C(X): |f(x)-h(x)|<\e \mbox{ for all } x\in F\}, \mbox{ where } F\in [X]^{<\w} \mbox{ and } \e>0,
\]
form a base at $h$ for the  topology  $\tau_p$
of pointwise convergence  on $C(X)$. The space $C(X)$
equipped with $\tau_p$ is usually denoted by $C_p(X)$.
\begin{lemma} \label{l:Cp-Ascoli-1}
Let $\{ U_n:n\in\w \}$ be a sequence of open subsets of $C_p(X)$ such that $0\in\overline{U_n}$ for all $n$. Then for every sequence   $\{\W_n:n\in\w\}$ such that $\W_n$ is an open cover of $U_n$, for every $n$ there exists $W_n\in\W_n$ such that $0\in\overline{\bigcup\{W_n:n\in\w\}}$.
\end{lemma}
\begin{proof}
By induction on $n$ we can construct an increasing sequence $\{ A_n:n\in\w \}$ of finite subsets of $X$, a decreasing null-sequence $\{\e_n:n\in\w\}$ of positive reals, a sequence $\{ W_n \in\W_n: n\in\w\}$ of open subsets of $X$  and a sequence $\{ h_n:n\in\w\}$ in $C_p(X)$ such that
\[
[h_n,A_n,\varepsilon_n]\subseteq W_n  \; \mbox{ and } \; [h_{n+1}, A_{n+1},\e_{n+1}] \subset [0, A_n, 1/n ].
 \]
We claim that $\{ W_n:n\in\w\}$ is as required. Indeed, fix a finite $F\subset X$ and $\e >0$, and find $n_0$ such that $F\cap(\bigcup_{n\in\w}A_n)\subset A_{n_0}$ and $\frac{1}{n_0}+\e_{n_0 +1}<\e$. Then any $h\in [h_{n_0+1},A_{n_0+1},\varepsilon_{n_0+1}]$ such that $h|_{ F\setminus A_{n_0 +1}} =0$ belongs to $[0,F,\e]$.
\end{proof}

The following statement is similar to \cite[Proposition~2.1]{GKP}.

\begin{lemma} \label{l:Cp-Ascoli-2}
Assume that $C_p(X)$ is an Ascoli space and $\{U_n:n\in\w\}$ is a sequence of open subsets of $C_p(X)$ such that $0\in\overline{\bigcup\{U_n:n\in\w\}}$
but $0\not\in\overline{U_n}$ for all $n$. Then there exists a compact subspace $K$ of $C_p(X)$ such that the set $\{n:K\cap U_n\neq\emptyset\}$ is infinite.
\end{lemma}
\begin{proof}
Suppose for a contradiction that for every compact $K\subset C_p(X)$,  $K\cap U_n\neq\emptyset$ only for finitely many $n$. For every $n\in\w$, set
\[
\W_n :=\left\{ W\in \mathcal P(U_n)\cap \tau_p : \exists \phi\in C\big(C_p(X)\big)\:(\phi|_{W}>1) \wedge (\phi|_{C_p(X)\setminus U_n}=0) \right\}.
\]
Then $\W_n$ is an open cover of $U_n$, and hence, by Lemma~\ref{l:Cp-Ascoli-1}, for every $n$ there exists $W_n\in\W_n$ such that $0\in\overline{\bigcup\{W_n:n\in\w\}}$. Let $\phi_n$ be a witness for $W_n\in\W_n$. It follows from the above that $\phi_n$ converges to $0$ in $\CC(C_p(X))$:
given any compact $K\subset C_p(X)$, $\phi_n|_K$ is constant $0$ for all but finitely many $n$ (namely for all $n$ such that $K\cap U_n=\emptyset$).
On the other hand, given any open $V\subset C_p(X)$ containing $0$ and $m\in\w$, the inclusion $0\in\overline{\bigcup\{W_n:n\in\w\}}$ implies that there exists $n\geq m$ and $f\in V\cap W_n$, which yields $\phi_n(f)>1$. This proves that the convergent sequence $$\{\phi_n:n\in\w\}\cup\{0\}\subset \CC(C_p(X))$$ is not evenly continuous, a contradiction.
\end{proof}

Following Arhangel'skii, a topological space $X$ is said to be {\em $\kappa$-Fr\'{e}chet--Urysohn} if for every open subset $U$ of $X$ and every $x\in \overline{U}$, there exists a sequence $\{ x_n\}_{n\in\NN} \subseteq U$ converging to $x$. Note that the class of $\kappa$-Fr\'{e}chet-Urysohn spaces is much wider than the class of Fr\'{e}chet--Urysohn spaces \cite{LiL}.

A family $\{ A_i\}_{i\in I}$ of subsets of a set $X$ is said to be {\em  point-finite} if the set $\{i\in I: x\in A_i\}$ is finite for every $x\in X$.
A family $\{ A_i\}_{i\in I}$ of subsets of a topological space $X$ is called
{\em strongly point-finite} if for every $i\in I$, there exists an open set
$U_i$ of $X$ such that $A_i\subseteq U_i$ and
 $\{ U_i\}_{i\in I}$ is point-finite.
 Following Sakai \cite{Sak2}, a topological space $X$ is said to have
{\em property $(\kappa)$} if every pairwise disjoint sequence of finite subsets
 of $X$ has a strongly point-finite subsequence.
We shall need the following result of Sakai,
see \cite[Theorem 2.1]{Sak2}.
\begin{theorem} \label{t:Sakai-Cp-k-FU}
 The space $C_p(X)$ is $\kappa$-Fr\'{e}chet--Urysohn if and only if $X$ has property $(\kappa)$.
\end{theorem}

Now we are ready to prove Theorem \ref{t:Cp-Ascoli-k-FU}.
\begin{proof}[Proof of Theorem \ref{t:Cp-Ascoli-k-FU}]
(i) By Theorem \ref{t:Sakai-Cp-k-FU} we have to show that $X$ has property $(\kappa)$. Consider a sequence $\{ F_n:n\in\w\}$ of finite subsets of $X$ such that $F_n\cap F_m=\emptyset$ for all $n\neq m$. We need to find  an infinite $J\subset \w$ and open sets $U_j\supset F_j$ for all $j\in J$, such that $\{U_j:j\in J\}$ is point-finite.

Let $g_k$ be the constant $k$ function and denote by $O_k$ the set $[g_k,F_k,1/2]$. It is easy to see that $0\in\overline{\bigcup\{O_k:k>0\}}$. By Lemma~\ref{l:Cp-Ascoli-2}  there exists a compact $K\subset C_p(X)$ intersecting infinitely many of the $O_k$'s. Thus there exists an infinite $J\subset \w$ and for every $j\in J$ a function $h_j\in K\cap O_j$. Set $$U_j :=\{x\in X: h_j(x)>j-1/2\}\supset F_j,$$ and note that $\{ U_j\}_{j\in J}$ is point-finite. Indeed, if $x$ belongs to $U_j$ for all $j\in J'$, where  $J'\subseteq J$ is infinite, then $\{h_j(x):j\in J'\}$ is unbounded, which is impossible because $\{h_j:j\in J'\}\subset K$.

\smallskip

(ii) Suppose that $C_p(X)$ is not Ascoli and find a compact $\KK\subset \CC(C_p(X))$ and $\phi\in\KK$ such that the valuation map is discontinuous at $(0,\phi)\in C_p(X)\times\KK$. Without loss of generality we may assume that $\phi(0)=0$ whereas the set
\[
\{(h,\psi)\in \big(C_p(X)\setminus\{0\}\big) \times\KK : \psi(h)>1\}
\]
contains $(0,\phi)$ in the closure. Let $\{ \Op_n:n\in\w\}$ be a base of the topology of $\KK$ at $\phi$.
For every $n\in\w$, denote by $H_n$ the set of all nonzero functions $h$ for which there is $\psi_{n,h}\in\Op_n$ such that $\psi_{n,h}(h)>1$, and note that $0\in\overline{H_n}$.
%It follows from the above that for every $n$ there exists $H_n\subset C_p(X)\setminus\{0\}$ such that $0\in\overline{H_n}$, and for every $h\in H_n$ an element $\psi_{n,h}\in\Op_n$ is such that $\psi_{n,h}(h)>1$.
Let $W_{n,h}\subset C_p(X)$ be an open neighbourhood of $h$ such that
\[
\overline{W_{n,h}}\subset C_p(X)\setminus\{0\}
\]
and $\psi_{n,h}(h')>1$ for all $h'\in W_{n,h}$. Set $\W_n=\{W_{n,h}:h\in H_n\}$ and  note that $0\in \overline{\bigcup\W_n} $ as $\bigcup\W_n\supset H_n$. Applying Lemma~\ref{l:Cp-Ascoli-1} we can find $h_n\in H_n$ such that $$0\in\overline{\bigcup\{W_{n,h_n}:n\in\w\}}.$$ Since $C_p(X)$ is $\kappa$-Fr\'echet--Urysohn there exists a convergent to $0$ sequence $\{ g_n:n\in\w\}$ such that $g_n\in W_{k_n,h_{k_n}}$ for some $k_n\in\w$.

Let $n_0$ be such that $\phi(g_n)<1/2$ for all $n\geq n_0$. Such an $n_0$ exists since $\phi$ is continuous and $\phi(0)=0$. Since $\{ \psi_{k_n,h_{k_n}}:n\in\w\}$ converges to $\phi$ in $\CC(C_p(X))$ and $\{g_n:n\in\w\}\cup\{0\}$ is a compact  subspace of $C_p(X)$,  there exists
$n_1\in\w$ such that
\[
\psi_{k_n,h_{k_n}}|_{ \{g_m:m\geq n_0\}\cup\{0\}} <1/2 \mbox{ for all } n\geq n_1.
\]
But this is impossible since $\psi_{k_n,h_{k_n}} (g_n)>1$ for all $n$, because  $g_n\in W_{k_n,h_{k_n}}$ and $\psi_{k_n,h_{k_n}} (h')>1$ for all  $h'\in W_{k_n,h_{k_n}}$.
\end{proof}

By  \cite[Theorem 3.2]{Sak2} every separable metrizable space $X$ with property $(\kappa)$ is always of the first category (i.e. every dense in itself subset $A$ of $X$ is of the first category in itself).
So, if $X$ is  a non-meager
separable metrizable space  without isolated points
 then $C_p(X)$ is not an Ascoli space.

Having in mind  Theorem \ref{t:Cp-Ascoli-k-FU} it is natural to ask the following
\begin{question}
Suppose that $C_p(X)$ is $\kappa$-Fr\'echet--Urysohn. Is it then Ascoli?%\footnote{\color{blue} I do not believe this by the following: the property $(\kappa)$ is hereditary by \cite[Proposition 3.7]{Sak2}. So it is enough to find $X$ and its subspace $Z$ such that $C_p(X)$ is Ascoli but $C_p(Z)$ is not. Then this $Z$ is as desired. Let me note that for metrizable $X$ the space $\CC(X)$ is Ascoli iff $X$ is locally compact. So we can wait that also for spaces $C_p(X)$ somthing similar is true (concerning hereditary).}
\end{question}

Next proposition complements Theorem 3.4 and Corollary 3.5 of \cite{Sak2}.
\begin{proposition} \label{p:Cech-complete-kappa-scat}
Let $X$ be a \v{C}ech-complete space. If $X$ is has property $(\kappa)$, then $X$ is scattered.
\end{proposition}

\begin{proof}
By Fact 1 on page 308 in \cite{Tkachuk-Book-1} it is enough to prove that any compact $K\subset X$ is scattered. Suppose for the contradiction that $X$ contains a non-scattered compact subset $K$. Since the property $(\kappa)$ is hereditary by \cite[Proposition 3.7]{Sak2}, to get a contradiction it is enough to show that $K$ does not have property $(\kappa)$. As $K$ is not scattered there exists a continuous surjective map $f:K\to [0,1]$, see \cite[Theorem~8.5.4]{sema}. By \cite[Exercise~3.1.C(a)]{Eng}, passing to the restriction of $f$ to some compact subspace of $K$ if necessary (this is possible because the property $(\kappa)$ is hereditary),  we may additionally  assume that $f$ is irreducible, i.e., $f[K']\neq [0,1]$ for any closed $K'\subsetneq K$. It follows that for any $A\subset K$, if $f[A]$ is dense in $[0,1]$, then $A$ is dense in $K$ because  $\overline{f[A]}=[0,1]$.

Let $\{B_n:n\in\w\}$ be a base of the topology of $[0,1]$. Since every $B_n$ is infinite we can choose a disjoint sequence $\{ F_n:n\in\w\}$ of finite subsets of $[0,1]$ such that $F_n\cap B_k\neq \emptyset$ for all $k\leq n$. Note that $\bigcup_{n\in I}F_n$ is dense in $[0,1]$ for every infinite subset $I$ of $\w$.
For every $n\in\w$ take a finite subset $A_n$ of $K$ such that $f[A_n]=F_n$. It follows from the above that $\bigcup_{n\in I}A_n$ is dense in $K$ for every infinite $I\subseteq \w$. We show that the sequence $\{ F_n: n\in\w\}$ does not have a strongly point-finite subsequence.
Let $I\subseteq\w$ be infinite and a sequence $\UU=\{ U_i: i\in I\}$ of open subsets of $K$ be such that $A_i\subseteq U_i$ for any $i\in I$. Then
\[
\bigcap_{m\in\w} \bigcup_{i\in I,i\geq m}U_i\neq \emptyset
\]
by the Baire theorem because $\bigcup_{i\in I,i\geq m}U_i$ is open and dense in $K$ for all $m$. So $\UU$ is not point-finite. Thus $K$ does not have property $(\kappa)$.
\end{proof}

Let $X=\prod_{ t\in T} X_t$ be the product of an infinite family of topological spaces. For $x=(x_t)$ and $y=(y_t)$ in $X$, we set $\delta(x,y):=\{ t: x_t\not= y_t\}$ and
\begin{equation} \label{equ-sigma}
\Sigma(x):= \{ y\in X: \delta(x,y) \mbox{ is countable}\} \mbox{ and } \sigma(x):= \{ y\in X: \delta(x,y) \mbox{ is finite}\}.
\end{equation}
If each $X_t$ is considered with a structure of a linear topological space,
then we standardly mean by  $\sigma_{t\in T} X_t :=\sigma(0)$ the  $\sigma$-product with
respect to the identity $0=0_X :=(0_t)\in X$. If $x\in \Sigma(z)$ we set
$\supp(x):=\{t\in T: x_t\not= z_t\}$, so $\supp(x)$ is a countable subset of $T$.
Subspaces of $\prod_{ t\in T} X_t$ of the form $\Sigma(x)$, where $x\in \prod_{ t\in T} X_t$,
are called \emph{$\Sigma$-subspaces.}

The following (probably folklore) statement generalizes a result of Noble \cite{Nob}. We give its proof for the sake of completeness.
\begin{proposition} \label{p:Noble-Sigma-product}
Let $\{X_i:i\in I\}$ be a family of topological spaces such that $X=\prod_{i\in I'}X_i$ is Fr\'echet--Urysohn for any countable subset $I'$ of $I$. Then $\Sigma(z)$ and hence also $\sigma(z)$ are Fr\'echet--Urysohn for every $z\in \prod_{ i\in I} X_i$.
In particular, each  $\Sigma$-subspace of a product of first countable spaces is a Fr\'{e}chet--Urysohn space.
\end{proposition}

\begin{proof}
Let $Z=\Sigma(z)$ be a $\Sigma$-subspace of $X=\prod_{ i\in I} X_i$. Fix $x_*\in Z$ and $A\subset Z$ such that $x_*\in\bar{A}$. Set $I_0:=\supp(x_*)$. Then $$\mathit{pr}_{I_0}(x_*)\in\overline{\mathit{pr}_{I_0}(A)},$$ and since $\prod_{i\in I_0} X_i$ is Fr\'echet--Urysohn, we can find a sequence $A_0 \subseteq A$ such that $$\mathit{pr}_{I_0}(x_*)\in\overline{\mathit{pr}_{I_0}(A_0)}.$$ Set $$I_1 :=I_0 \cup \bigcup_{x\in A_0} \supp(x).$$ Repeating the above arguments by induction on $n\in\w$, we construct countable sets $I_n\subset I$ and $A_n\subset A$ with the following properties for all $n\in\w$:
\begin{itemize}
 \item[(a)]   $I_n\subset I_{n+1}$, and $A_n\subset A_{n+1}$;
 \item[(b)] $\mathit{pr}_{I_n}(x_*)\in\overline{\mathit{pr}_{I_n}[A_n]}$; and
 \item[(c)] $\bigcup\{\supp(x):x\in A_n\}\subset I_{n+1}$.
 \end{itemize}

We claim that $x_*\in\overline{A_\w}$, where $A_\w=\bigcup_{n\in\w}A_n$. Indeed, fix a finite $F\subset I$ and open $U_i\ni x_*(i)$ for all $i\in F$. Set $$I_\w=\bigcup_{n\in\w} I_n$$ and find $n_0$ such that $$F\cap I_\w=F\cap I_{n_0}.$$ By our choice of $A_{n_0}$ there exists $x\in A_{n_0}$ such that $x(i)\in U_i$ for all $i\in F\cap I_\w$. Moreover, for $i\in F\setminus I_\w$ we have that $x_*(i)=x(i)=z(i)$ because $\supp(x_*)\subset I_0$ and $\supp(x)\subset I_{n_0+1}$, and hence $x(i)\in U_i$ for all $i\in F$, which yields  $x_*\in\overline{A_\w}$.

The space $\prod_{i\in I_\w} X_i$ is Fr\'{e}chet--Urysohn, and therefore there exists a sequence $\{ x_n:n\in\w\}$ of elements of $A_\w$ such that the sequence $\{\mathit{pr}_{I_\w}(x_n):n\in\w\}$ converges to $\mathit{pr}_{I_\w}(x_*)$ in $\prod_{i\in I_\w}X_i$. Since $x_n(i)=z(i)$ for all $n\in\w$ and $i\in I\setminus I_\w$, we conclude that  $x_n \to x_*$ in $Z$. Thus $Z$ is a Fr\'{e}chet--Urysohn space.
\end{proof}

In what follows we need the following consequence  of  \cite[Proposition 5.10]{BG}.
\begin{lemma} \label{p:Ascoli-dense-subgroup}
Let $Y$ be a dense subset of a homogeneous space (in particular, a topological group) $X$. If $Y$ is an Ascoli space, then $X$ is also an Ascoli space.
\end{lemma}
\begin{proof}
Fix arbitrarily $y_0\in Y$. Let $x\in X$. Take a homeomorphism $h$ of $X$ such that $h(y_0)=x$. Then $x\in h(Y)$ and $h(Y)$ is an Ascoli space. So each element of $X$ is contained in a dense Ascoli subspace of $X$.  Thus $X$ is an Ascoli space by Proposition 5.10 of \cite{BG}.
\end{proof}

Let us recall several definitions. For a scattered space $X$ one of the most efficient methods to analyze its structure is the Cantor--Bendixson procedure described below. Set $X^{(0)}:=X$,
\[
X^{(\gamma+1)}:=X^{(\gamma)}\setminus\mathit{Iso}(X^{(\gamma)})
\]
(where by $\mathit{Iso}(Z)$ we denote the set of all isolated points of a space $Z$), and $$X^{(\gamma)}:=\bigcap_{\alpha<\gamma}X^{(\alpha)}$$ for limit ordinals $\gamma$. It is easy to see that $X$ is scattered if and only if $X^{(\gamma)}=\emptyset$ for some ordinal $\gamma$. If $X$ is scattered, for $x\in X$ we denote by $d(x)$ the (unique) $\alpha$ such that $x\in X^{(\alpha)}\setminus X^{(\alpha+1)}$.

A space $X$ is  called \emph{ultraparacompact} \cite{RudWat83} if any open cover has a clopen disjoint refinement. It has been shown by Telgarsky in \cite{Tel68} that a scattered paracompact space is zero-dimensional and ultraparacompact, see also \cite{RudWat83} for generalizations.

\begin{proposition} \label{p:Cp-Ascoli-czech-complete-paracomp}
Assume that a paracompact scattered space $X$ has the following property:
\begin{itemize}
 \item[$(\star)$]
Each $x\in X$ has a clopen neighborhood $O(x)$ such that for any clopen $U$, $x\in U\subset O(x)$,  there exists a compact $C$, $x\in C\subset U$, for which the difference $U\setminus C$ is paracompact, and there exists a continuous linear operator $\psi:C_p(C)\to C_p(U)$ such that $\psi(f)|_C=f$ for all $f\in C_p(C)$.
\end{itemize}
 Then $C_p(X)$ is Ascoli.
\end{proposition}

\begin{proof}
Note that if $x$ and its clopen neighborhood $O(x)$ satisfy $(\star)$, then for every clopen neighborhood $V$ of $x$ with $V\subseteq O(x)$ the pair $x,V$ satisfies $(\star)$. The following claim is the central part of the proof.

\begin{claim} \label{claim:Cp-scattered}
Let $X$ be a  paracompact scattered space with property $(\star)$. Then for every $x\in X$ there exists a clopen neighbourhood $O(x)$ of $x$ with the following property:
\begin{itemize}
 \item[$(\dagger)$] For any clopen $U\subset O(x)$ there exists a family $\KK=\KK_{U,x}$ of scattered compact
 subsets of $U$ such that $C_p(U)$ is linearly homeomorphic to a linear subspace of $\prod_{K\in\KK}C_p(K)$ containing $\sigma_{K\in\KK}C_p(K)$.
\end{itemize}
\end{claim}

\begin{proof}
The proof will be by transfinite induction on $d(x)$. If $d(x)=0$, then $x\in \mathit{Iso}(X)$. Set $O(x):=\{x\}$ and $\KK :=\big\{\{x\}\big\}$. Clearly, $O(x)$ and $\KK$ are as required. Assuming that the claim is true for all $x\in X$ with $d(x)<\alpha$, let us fix $x\in X$ with $d(x)=\alpha$ and find a clopen neighborhood $O(x)\subseteq X$ of $x$ such that
\[
O(x)\subseteq \{y\in X:d(y)< \alpha\} \cup\{x\}.
\]
We claim that $O(x)$ is as required. Indeed, let us fix a clopen $U\subseteq O(x)$. Two cases are possible.

{\em Case 1. Assume that $x\in U$.} Since $X$ has property $(\star)$, there exists a compact $C\ni x$ such that  $C\subset U$ and $U\setminus C$ is
paracompact and  there exists a continuous linear operator $\psi:C_p(C)\to C_p(U)$ such that $\psi(f)|_C=f$ for all $f\in C_p(C)$, so $\psi(0)=0$.  For every $y\in U\setminus C$, set $$V_0(y)=O(y)\cap (U\setminus C).$$ Then $\V_0=\{V_0(y):y\in U\setminus C\}$ is an open cover of a paracompact scattered space $U\setminus C$. Thus there exists  \cite{Tel68} a clopen cover $\V$ of $U\setminus C$ whose elements are mutually disjoint, and such that $\V\prec\V_0$, i.e., for every $V\in\V$ there exists $V'\in\V_0$ with the property $V\subset V'$. It follows from the above that each $V\in\V$ has property $(\dagger)$, and hence there exists a family $\KK_{V}$ of scattered compact subsets of $V$ such that $C_p(V)$ can be topologically embedded into
$\prod_{K\in\KK_{V}}C_p(K)$ via a linear continuous map
\[
\phi_{V}:C_p(V)\to \prod_{K\in\KK_{V}}C_p(K)
\]
 such that
 \[
\sigma_{K\in\KK_{V}}C_p(K) \subset   \phi_{V}[C_p(V)].
\]
 Set
\[
\KK=\bigcup \{\KK_{V}: V\in\V\}\cup\{C\},
\]
so $\KK$ is a family of scattered compact subsets of $U$. Define a continuous linear operator $\phi:C_p(U)\to\prod\{C_p(K):K\in \KK\}$  as follows: if $f\in C_p(U)$, then
\begin{equation} \label{equ:Cp-Ascoli-Claim-0}
\phi(f)(C)=f|_C;
\end{equation}
and if $K\in \KK_V$ for the unique $V\in\V$ such that $K\in\KK_V$, then
\begin{equation} \label{equ:Cp-Ascoli-Claim-1}
\phi(f)(K)=\phi_{V}\big((f-\psi(f|_C))|_V\big)(K).
\end{equation}
In (i)-(iii) below we prove that $\phi$ and $\KK$ satisfy $(\dagger)$.

\smallskip
{\em (i) We show that $\sigma_{K\in\KK}C_p(K)\subset \phi[C_p(U)]$.} Fix a finite $\KK'\subset \KK$ and
\[
(f_K)_{K\in\KK'}\in\prod_{K\in\KK'}C_p(K).
 \]
There is no loss of generality to assume that $C\in \KK'$, because otherwise we may consider $\KK''=\KK' \cup\{C\}$ and set $f_C =0$. For every $K\in\KK'\setminus\{C\}$ find (the unique) $V_K\in\V$ such that $K\in\KK_{V_K}$. For every $V\in\{V_K:K\in\KK'\}$ find $f_V\in C_p(V)$ such that for each $K\in\KK'$ with $V_K=V$ it follows that
\begin{equation} \label{equ:Cp-Ascoli-Claim-10}
\phi_V(f_V)(K)=f_K.
\end{equation}
Such an $f_V$ exists by our assumptions on $\phi_V$. Set
\[
U':=U\setminus\bigcup\{V_K:K\in\KK'\setminus\{C\}\},
\]
so $U'$ is a clopen subset of $X$ containing $C$. Define $f\in C_p(U)$ by
\begin{equation} \label{equ:Cp-Ascoli-Claim-2}
f(x) := \left\{
\begin{aligned}
\psi(f_C)(x) \quad \quad , & \mbox{ if } x \in U',\\
\psi(f_C)(x) + f_{V_K}(x), & \mbox{ if } x\in V_K \mbox{ and } K\in \KK'\setminus\{C\}.
\end{aligned}
\right.
\end{equation}
We claim that $\phi(f)(K)$ equals $f_K$ for $K\in\KK'$ and $0$ otherwise, that proves (i). Indeed, fix $K\in\KK'$. If $K=C\subseteq U'$, then
\[
\phi(f)(C) \stackrel{(\ref{equ:Cp-Ascoli-Claim-0})}{=} f|_C \stackrel{(\ref{equ:Cp-Ascoli-Claim-2})}{=}  \psi(f_C)|_C=f_C.
\]
If $K\in\KK'\setminus\{C\}$, then
\[
\begin{split}
\phi(f)(K) & \stackrel{(\ref{equ:Cp-Ascoli-Claim-1})}{=}  \phi_{V_K}\big((f-\psi(f|_C))|_{V_K}\big)(K) = \phi_{V_K}\big(f|_{V_K}-\psi(f_C)|_{V_K}\big)(K)\\
& \stackrel{(\ref{equ:Cp-Ascoli-Claim-2})}{=}  \phi_{V_K}\big((\psi(f_C)|_{V_K} + f_{V_K})-\psi(f_C)|_{V_K}\big)(K)= \phi_{V_K}\big(f_{V_K}\big)(K) \stackrel{(\ref{equ:Cp-Ascoli-Claim-10})}{=} f_K.
\end{split}
\]
Finally, if $K\in\KK\setminus \KK'$, then $K\subseteq U'$ and
\[
\begin{split}
\phi(f)(K) & \stackrel{(\ref{equ:Cp-Ascoli-Claim-1})}{=}  \phi_{V_K}\big((f-\psi(f|_C))|_{V_K}\big)(K) = \phi_{V_K}\big(f|_{V_K}-\psi(f_C)|_{V_K}\big)(K)\\
& \stackrel{(\ref{equ:Cp-Ascoli-Claim-2})}{=}  \phi_{V_K}\big(\psi(f_C)|_{V_K} -\psi(f_C)|_{V_K}\big)(K)= \phi_{V_K}\big(0\big)(K)=0.
\end{split}
\]

\smallskip
{\em (ii) Let us prove that $\phi$ is injective.} Assume that $\phi(f)=\phi(g)$. Set  $h:=\phi(f)(C)= f|_C= g|_C$. Given any $V\in\V$ and $K\in\KK_V$, the equality $\phi(f)(K)=\phi(g)(K)$ and (\ref{equ:Cp-Ascoli-Claim-1}) imply
\[
\phi_{V}\big((f-\psi(h))|_V\big)(K)= \phi_{V}\big((g-\psi(h))|_V\big)(K),
\]
and hence $(f-\psi(h))|_V=(g-\psi(h))|_V$ by the injectivity of $\phi_V$. Consequently, $f|_V= g|_V$, and therefore $f=g$ because $V\in\V$ was chosen
arbitrarily.

\smallskip
{\em (iii) We show that $\phi^{-1}:\phi[C_p(U)]\to C_p(U)$ is continuous.} Fix a finite subset $F$ of $U$ and $\e>0$. Passing to a larger  $F$ if necessary we may assume that $F=F_C\cup\bigcup\{F_i:i\leq n\}$, where $F_C\in [C]^{<\w}$ and $F_i\in [V_i]^{<\w}$ for some $V_i\in \V$ such that $V_i\neq V_j$ for $i\neq j$. We need to find an open neighbourhood $W$ of
\[
(0_K)\in\prod_{K\in\KK}C_p(K)
\]
 such that $f\in [0,F,\varepsilon]$ whenever $\phi(f)\in W$. Let $A_C\in [C]^{<\w}$  and $\delta >0$ be such that $F_C\subset A_C$, $\delta<\e$, and
\[
 \psi[0, A_C,\delta]\subset [0,F,\e/2].
\]
(Here of course $[0, A_C,\delta]$ and $[0,F,\e/2]$ are considered as subsets of $C_p(C)$ and $C_p(U)$, respectively). Since $\phi_{V_i}$ is an embedding, there exists an open neighbourhood $W_i$ of
\[
(0_K)\in \prod_{K\in\KK_{V_i}}C_p(K)
\]
 such that $h\in C_p(V_i)$ lies in $[0,F_i,\e/2]$ whenever $\phi_{V_i}(h)\in W_i$. Consider
\[
W=W_C\times\prod_{V\in\V} W_V
\]
such that
\[
W_C=[0, A_C,\delta]\subset C_p(C),\,\,\,W_{V_i}=W_i\subset\prod_{K\in\KK_{V_i}}C_p(K),
\]
and $W_{V}=\prod_{K\in\KK_{V}}C_p(K)$ for $V\not\in\{V_i:i\leq n\}$. Assume that $\phi(f)\in W$ for some $f\in C_p(U)$. Then
\[
\phi(f)(C)=f|_C\in  [0,A_C,\delta]\subset [0,F_C,\e],
\]
and hence $\psi(f|_C)|_{V_i}\in [0,F_i,\e/2]$ for all $i\leq n$. Fix $i\leq n$ and observe that $\phi(f)\in W$ implies
$\phi(f)\uhr\KK_{V_i}\in W_i$; therefore, see also (\ref{equ:Cp-Ascoli-Claim-1}), $\phi_{V_i}(h_i)\in W_i$  for $h_i=(f-\psi(f|_C))|_{V_i}$. It follows from the above that $h_i \in [0,F_i,\e/2]\subset C_p(V_i)$. Since $\psi(f|_C))|_{V_i}\in [0,F_i,\e/2]$ and $h_i \in [0,F_i,\e/2]$, we have that
\[
f|_{V_i}=h_i+ \psi(f|_C))|_{V_i}\in [0,F_i,\e]
\]
which completes our proof in Case 1.

\smallskip
{\em Case 2. Assume that $x\not\in U$.}  This case is similar but simpler than the previous one.  Given any $y\in U$, set  $V_0(y)=O(y)\cap U $. Then $\V_0=\{V_0(y):y\in U\}$ is an open cover of a paracompact  scattered space $U$. So there exists a clopen cover $\V\prec \V_0$ of $U$ whose elements are mutually disjoint, see  \cite{Tel68}.  It follows from the above that each $V\in\V$ has property $(\dagger)$, and hence there exists a family $\KK_{V}$ of scattered compact subsets of $V$ such that $C_p(V)$ can be topologically embedded into $\prod_{K\in\KK_{V}}C_p(K)$ via a linear continuous map $$\phi_{V}:C_p(V)\to \prod_{K\in\KK_{V}}C_p(K)$$ such that $$\phi_{V}[C_p(V)]\supset\sigma_{K\in\KK_{V}}C_p(K).$$ Set  $\KK :=\bigcup \{\KK_{V}: V\in\V\}$ and
\[
\phi=(\phi_V)_{V\in\V} : C_p(U)=\prod_{V\in\V} C_p(V) \to \prod_{V\in\V} \prod_{K\in \KK_V}  C_p(K) =\prod \{C_p(K):K\in \KK\}.
\]
A direct verification shows that $\phi$ is a linear  embedding and $\phi[C_p(U)]$ contains $\sigma_{K\in \KK}C_p(K)$.
\end{proof}
Now we complete the proof of the proposition.  By Claim \ref{claim:Cp-scattered}, for every $x\in X$ choose a clopen neighbourhood $O(x)$ of $x$ with the property $(\dagger)$. Then $\V_0=\{O(x):x\in X\}$ is an open cover of a paracompact  scattered space $X$. By the same argument as in the proof of Case 2 of
Claim~\ref{claim:Cp-scattered} we get that  there exists a family $\KK$ of scattered compact spaces such that $C_p(X)$ is linearly homeomorphic to a linear subspace of $\prod_{K\in\KK}C_p(K)$ containing $\sigma_{K\in\KK}C_p(K)$. The latter $\sigma$-product is dense in $C_p(X)$ as it is dense in  $\prod_{K\in\KK}C_p(K)$. For any countable $\KK'\subset\KK$ the topological sum $\oplus\KK'$ is a Lindel\"of scattered space, and hence $$\prod_{K\in\KK'}C_p(K)=C_p(\oplus\KK')$$ is Fr\'echet--Urysohn by \cite[Theorem~II.7.16]{Arhangel}.  So $\sigma_{K\in\KK}C_p(K)$ is Fr\'echet--Urysohn by
Proposition~\ref{p:Cech-complete-kappa-scat},  and hence $C_p(X)$ can be covered by its dense Fr\'echet--Urysohn subspaces (namely shifts of $\sigma_{K\in\KK}C_p(K)$). Thus $C_p(X)$ is Ascoli by Lemma \ref{p:Ascoli-dense-subgroup}.
\end{proof}

Clearly if  $X$ has finitely many non-isolated points, then $X$ has property $(\star)$. Therefore we have the following
\begin{corollary} \label{c:Cp-Ascoli-finite}
If  $X$ has finitely many non-isolated  points then $C_p(X)$ is Ascoli.
\end{corollary}

A regular topological space $X$ is {\em stratifiable} if there is a function $G$ which assigns to every $n\in\w$ and each closed set $F\subset X$ an open neighborhood $G(n,F)\subset X$ of $F$ such that $F=\bigcap_{n\in\w} \overline{G(n,F)}$ and $G(n,F)\subset G(n,F')$ for any $n\in\w$ and closed sets $F\subset F'\subset X$. Borges proved in \cite{Bor} that each stratifiable space $X$ satisfies Dugundgji's extension theorem: For every closed subset $A$ of $X$ there is a continuous linear operator $\psi:\CC(A)\to \CC(X)$ such that $\psi(g)|_A =g$ for every $g\in \CC(A)$.  Any metrizable space is stratifiable, and each stratifiable space is paracompact, see \cite[Theorem~5.7]{gruenhage}. Any subspace of a stratifiable space is stratifiable and hence is paracompact.

\begin{proof}[Proof of Theorem \ref{t:Cech-complete-Ascoli-scat}]
(i) follows from  Theorems \ref{t:Sakai-Cp-k-FU} and \ref{t:Cp-Ascoli-k-FU} and Proposition \ref{p:Cech-complete-kappa-scat}.

(ii) By Proposition \ref{p:Cp-Ascoli-czech-complete-paracomp} it is enough to show that every scattered  stratifiable space satisfies property $(\star)$. For every $x\in X$, let $O(x)$ be an arbitrary clopen neighborhood of $x$ and let $C=\{ x\}$. Now for every clopen $U$ with $x\in U\subseteq O(x)$, the difference $U\setminus C$ is paracompact  and there is a continuous linear operator $\psi: \CC(C)\to \CC(U)$. At the end of page 9 in \cite{Bor}  Borges proved that the operator $\psi$ is also continuous as a map from $C_p(C)$ to $C_p(U)$. Thus $X$ satisfies property $(\star)$.
\end{proof}

In light of Theorem~\ref{t:Cech-complete-Ascoli-scat} it is natural to ask the
following
\begin{question}
Does every scattered \v{C}zech-complete space have property $(\star)$?
\end{question}

%Now Theorems \ref{t:Sakai-Cp-k-FU} and \ref{t:Cp-Ascoli-k-FU} and Proposition \ref{p:Cech-complete-kappa-scat} imply
%\begin{corollary} \label{c:Cech-complete-Ascoli-scat}
%Let $X$ be a \v{C}ech-complete space. If $C_p(X)$ is Ascoli, then $X$ is scattered.
%\end{corollary}

The following corollary complements Theorem II.7.16 of \cite{Arhangel} and immediately implies Corollary \ref{c:compact-Ascoli-scat}.
\begin{corollary} \label{c:Cp-Ascoli-FU}
%(A)
For a \v{C}ech-complete Lindel\"{o}f space $X$, the following assertions are equivalent:
\begin{itemize}
 \item[(i)] $C_p(X)$ is Ascoli;
  \item[(ii)] $C_p(X)$ is Fr\'echet--Urysohn;
  \item[(iii)] $X$ is scattered;
\item[(iv)] $X$ is scattered and $\sigma$-compact.
\end{itemize}
%(B) Let $X$ be a locally compact  space.
%Then $C_{p}(X)$ is Asoli if and only if $X$ is scattered.
\end{corollary}

\begin{proof}
%Part (A):
(i)$\Rightarrow$(iii) follows from (i) of Theorem \ref{t:Cech-complete-Ascoli-scat}, (iii)$\Rightarrow$(ii) follows from \cite[Theorem~II.7.16]{Arhangel}, and (ii)$\Rightarrow$(i) is trivial.
(iii) $\Rightarrow$ (iv): If $X$ is a  \v{C}ech-complete Lindel\"{o}f space, then by Frolik's theorem, see \cite{comfort}, there exists a Polish space $Y$  and a perfect map from $X$ onto $Y$. As scattered property is inherited by perfect maps, the space $Y$ is scattered, hence countable by \cite[8.5.5]{sema}. Consequently $X$ is $\sigma$-compact.
\end{proof}

The famous Pytkeev--Gerlitz--Nagy theorem, see  \cite[Theorem~II.3.7]{Arhangel}, states that
  $C_p(X)$ is  a $k$-space if and only if $C_p(X)$ is Fr\'{e}chet--Urysohn if and only if
 $X$ has the covering property $(\gamma)$ introduced in \cite{GerNagy}.
 Below we give an example of a separable metrizable space $X$ for
which $C_p(X)$ is Ascoli but is not a $k$-space.  So the property to be an Ascoli space
is strictly weaker than the property to be a $k$-space for $C_p(X)$
even in the class of separable metric spaces.

 Recall that a separable metric space $X$ is said to be a {\em $\lambda$-space}
if every countable subset of $X$ is a $G_\delta$-set of $X$.
Every $\lambda$-space has property $(\kappa)$ by \cite[Theorem 3.2]{Sak2}.
So $C_p(X)$ is Ascoli by Corollary \ref{c:Cp-k-FU-metr} for such space $X$.

\begin{example} \label{exa:Cp-Ascoli-non-k-space} {\em
Rothberger proved in \cite{Rothberger} that there is an unbounded subset $X$ of
$\w^\w$ which is a $\lambda$-space, see also
 \cite[p.~215]{Miller}.
 So $X$ is a separable metrizable space with property $(\kappa)$ by Theorem 3.2 of
 \cite{Sak2}. Therefore $C_p(X)$ is an Ascoli space by Theorem \ref{t:Sakai-Cp-k-FU}
and Corollary \ref{c:Cp-k-FU-metr}. However, it follows from the results of
Gerlits and Nagy \cite{GerNagy} that no unbounded
subset of $\omega^\omega$ has property $(\gamma)$, and hence $C_p(X)$ is not
 Fr\'{e}chet--Urysohn. %However, Hurewicz proved in \cite{Hurewicz} that $X$ does not have property $(\gamma)$ (see also Theorem 4.3 of \cite{JMSS}).
So $C_p(X)$ is not a $k$-space by the Pytkeev--Gerlitz--Nagy theorem. }
\end{example}

\begin{question} \label{q:Cp-cosmic-k_R}
Let $X$ be an uncountable cosmic space such that $C_p(X)$ is Ascoli (for example, $X$ is a  $\lambda$-space). Is then $C_p(X)$ a $k_\IR$-space?
\end{question}
The negative answer to this question would give an example of an Ascoli space $C_p(X)$ for separable  metrizable $X$ which is not a $k_{\IR}$-space. Let us note that the example provided in \cite{GGKZ-2} is not metrizable.

The assumption to be \v{C}ech-complete is essential for the results of this section as the metrizable space $C_p(\mathbb{Q})$ shows. We end this section with the following question.
\begin{question} \label{q:Cp-metriz}
For which metrizable spaces $X$ the space $C_p(X)$ is Ascoli?
\end{question}

%%%%%%%%%%%%%%%%%%%%%%%%%%%%%%%%%%%%%%%%
%%%%%%%%%%%%%%%%%%%%%%%%%%%%%%%%%%%%%%%%
%%%%%%%%%%%%%%%%%%%%%%%%%%%%%%%%%%%%%%%%
%%%%%%%%%%%%%%%%%%%%%%%%%%%%%%%%%%%%%%%%
%%%%%%%%%%%%%%%%%%%%%%%%%%%%%%%%%%%%%%%%

\section{The Ascoli property for $\CC(X)$} \label{sec:Ck}

%%%%%%%%%%%%%%%%%%%%%%%%%%%%%%%%%%%%%%%%
%%%%%%%%%%%%%%%%%%%%%%%%%%%%%%%%%%%%%%%%
%%%%%%%%%%%%%%%%%%%%%%%%%%%%%%%%%%%%%%%%
%%%%%%%%%%%%%%%%%%%%%%%%%%%%%%%%%%%%%%%%
%%%%%%%%%%%%%%%%%%%%%%%%%%%%%%%%%%%%%%%%

Let $X$ be a Tychonoff space and $\KK(X)$ be the set of all compact subsets of $X$. For $h\in C(X)$  the sets of the form
\[
[h,K,\e]:= \{ f\in C(X): |f(x)-h(x)|<\e \mbox{ for all } x\in K\}, \mbox{ where }
K\in \mathcal K(X)  \mbox{ and } \e>0,
\]
form a base at $h$ for the \emph{compact-open} topology  $\tau_k$  on $C(X)$. The space $C(X)$ equipped with $\tau_k$ is usually denoted by $\CC(X)$.

Theorem 2.5 of \cite{GKP} states in particular that, for a first-countable paracompact $\sigma$-space $X$, the space $\CC(X)$ is an Ascoli space if and only if  $\CC(X)$ is a  $k_\IR$-space if and only if $X$ is a locally compact metrizable space. In this section we prove an
 analogous result  using the following proposition.
\begin{proposition}[\cite{GKP}] \label{p:Ascoli-sufficient}
Assume   $X$ admits a  family $\U =\{ U_i : i\in I\}$ of open subsets of $X$, a subset $A=\{ a_i : i\in I\} \subset X$ and a point $z\in X$ such that
\begin{enumerate}
\item[{\rm (i)}] $a_i\in U_i$ for every $i\in I$;
\item[{\rm (ii)}] $\big|\{ i\in I: C\cap U_i\not=\emptyset \}\big| <\infty$  for each compact subset $C$ of $X$;
\item[{\rm (iii)}] $z$ is a cluster point of $A$.
\end{enumerate}
Then $X$ is not an Ascoli space.
\end{proposition}

Recall that  $X$  is of {\em point-countable type} if for every  $x\in X$ there
exists a compact $K$ containing $x$ such that $K$ has  countable basis of
neighborhoods, i.e. there is a sequence of open sets $\{U_n\}_{n<\w}$ such
 that $K\subseteq U_n$ for all $n<\w$ and for every open $O$ containing $K$
 there is $n<\w$  such that $U_n\subseteq O$. The following statement is
reminiscent of \cite[Proposition~2.3]{GKP}, and substantially uses the idea of
 R.~Pol from \cite{Pol-1974}.  We say that a space $X$
is \emph{locally pseudocompact} if for every $x\in X$ there exists an open $U\ni x$
whose closure $\bar{U}$ is pseudocompact.

\begin{lemma} \label{l:Ascoli-loc_pseud_comp}
Let $X$ be a  space of point-countable type. If $C_k(X)$ or $\CC(X,\II)$ is an Ascoli space, then $X$ is locally pseudocompact.
\end{lemma}

\begin{proof}
Assume that $X$ is not locally pseudocompact, so there exists $x_0\in X$ such that no neighborhood of $x_0$ is pseudocompact. Because $X$ is of point-countable type there is a compact set $K\subset X$ such that $x_0\in K$ and there is a base of neighborhoods $\{U_n\}_{n\in\w}$ of $K$ such that $\overline{U_{n+1}}\subsetneq U_n$  (here we use the fact that $K$ is compact and $X$ is Tychonoff).

We show that there is a strictly increasing sequence $\{n_k\}_{k\in \w}$ such that $n_{k+1}>n_k+1$ and for every $k\in\w$, the difference
$\overline{U_{n_k}}\setminus U_{n_k+1}$ is not pseudocompact. Indeed, otherwise there exists $n_0$ such that $\overline{U_{n}}\setminus U_{n+1}$ is pseudocompact for all $n\geq n_0$. We claim that  $\overline{U_{n_0}}$ is a pseudocompact neighborhood of $x_0$  which leads to a contradiction. Given any continuous $f:\overline{U_{n_0}}\to \IR$, there exists $m\in \IR$ such that $f^{-1}[(-m,m)]$ is an open set containing $K$, and therefore it contains some $U_{n_1}$, which together with the pseudocompactness of $\overline{U_{n_0}}\setminus U_{n_1}$ implies that $f$ is bounded.

Set $P_k:=\overline{U_{n_k}}\setminus U_{n_k+1}$. Since every $P_k$ is not pseudocompact, by \cite[Theorem~3.10.22]{Eng}  there exists a locally finite collection $\{U_{i,k}\}_{i<\w}$ of nonempty open subsets  of $P_k$. We may assume in addition that  every $U_{i,k}\subseteq \mathrm{Int}(P_k)$. Pick any $x_{i,k}\in U_{i,k}$, and for $1\leq k<i$ find continuous functions $f_{i,k} :X\to [0,1]$ such that
\[
f_{i,k}(x_{i,k})=1, \ f_{i,k}(x_{i,i})=0,  \mbox{ and } \ f_{i,k}(x) = \frac{1}{k} \mbox{ for } \ x\not\in U_{i,k}\cup U_{i,i}.
\]
Set $A:=\{f_{i,k}:1\le k<i<\w\}$ and $\V :=\{V_{i,k}\}_{1\leq k<i<\w}$, where $V_{i,k}\subset \CC(X)$ or $V_{i,k}\subset \CC(X,\II)$ and $h\in V_{i,k}$ if
\[
|h(x_{i,k})-1|<\frac{1}{4^{i+k}}, \ |h(x_{i,i})|<\frac{1}{4^{i+k}}, \mbox{ and } \left| h(x)-\frac{1}{k}\right|<\frac{1}{4^{i+k}} \mbox{ for all }x\in K.
\]
We shall complete the proof by  showing  that $A$, $\V$ and $0$ satisfy the assumption of Proposition~\ref{p:Ascoli-sufficient}. The first one is by definition. For (iii),  assume that $Z\subset X$ is compact  and fix $\e>0$. Find $k<\w$ such that $\frac{1}{k}<\e$ and $i>k$ such that $Z\cap U_{i,k} = \emptyset$ (this is possible because $Z$ is compact and $\{U_{i,k}\}_{i<\w}$ is a locally finite collection).  It follows  that $$f_{i,k}(z)\leq \frac{1}{k}<\e$$ for every $z\in Z$. Thus $0\in \overline{A}$.

Let us check (ii): any compact subset $C$ of $\CC(X)$ or of $\CC(X,\II)$ meets only finitely many elements of $\V$. By the Ascoli theorem \cite[Theorem~3.4.20]{Eng}, for every compact $Z\subset X$,  $x\in Z$ and $\e>0$ there is a  neighborhood $O_x$ of $x$ such that $|f(x)-f(y)|<\e$ for all $y\in O_x\cap Z$ and $f\in C$. Define
\[
Z_0:=\{x_{i,k}:1\leq i\leq k<\w\}\cup K,
\]
and note that $Z_0$ is a compact subset of $X$.

We claim that for every $k<\w$ there is $i_0>k$ such that $C\cap V_{i,k}=\emptyset$ for every $i>i_0$.
Indeed, assume the converse. Using  the Ascoli theorem for $C$, $Z_0$ and $\e=\frac{1}{3k}$, for every $x\in Z_0$  we find a neighborhood $O_x$ of $x$ such that $|h(y)-h(x)|<\e$ for every $y\in O_x$ and $h\in C$. Then the collection $\{O_x\}_{x\in Z_0}$ covers $K\subset Z_0$, so there exists $i_0$ such that $U_{i_0}\subset \bigcup_{x\in Z_0} O_x$. Take any $i>i_0$, $h\in C\cap V_{i,k}$ and  $x\in K$ such that $x_{i,i}\in O_x$ (recall that $x_{i,i}\in U_{i,i} \subset P_{i} \subset U_{n_{i_0}}$, and clearly $n_{i_0}\geq i_0$). By construction, $$K\cap (U_{i,k}\cup U_{i,i}) =\emptyset,$$ so $f_{i,k}(x)=1/k$ and $f_{i,k}(x_{i,i})=0$. Since $h\in C\cap V_{i,k}$ we obtain
\begin{eqnarray*}
\begin{aligned}
\frac{1}{3k} > |h(x_{i,i})-h(x)| & \geq |f_{i,k}(x_{i,i})-f_{i,k}(x)|-|f_{i,k}(x_{i,i})-h(x_{i,i})|-|h(x)-f_{i,k}(x)| \\
& >  \frac{1}{k}-\frac{1}{4^{i+k}}-\frac{1}{4^{i+k}}>\frac{1}{3k},
\end{aligned}
\end{eqnarray*}
 a contradiction. This contradiction proves the claim.

To finish the proof it is enough to show that there is no sequence $\{(i_n,k_n)\}_{n<\w}$ such that
\[
...<k_n<i_n<k_{n+1}<i_{n+1}<...
\]
and $V_{i_n,k_n}\cap C\not=\emptyset$. If not, consider the compact subset $Z_1:=\{x_{i_n,k_n}:n<\w\}\cup K$ of $X$. Using  the Ascoli theorem for $C$, $Z_1$ and $1/3$, for every $x\in Z_1$  we find a neighborhood $O_x$ of $x$ such that $|h(y)-h(x)|<1/3$ for every $y\in O_x$ and $h\in C$. Again, the collection $\{O_x\}_{x\in K}$ covers $K\subset Z_1$, so there exists $k>10$ such that $U_{k}\subset \bigcup_{x\in K} O_x$. Pick $n$ such that $k<k_n$ and note that
there is $x\in K$ such that $x_{i_n,k_n}\in O_x$.  Then, as above, for any $h\in V_{i_n,k_n}\cap C$ we have
\begin{eqnarray*}
\begin{aligned}
\frac{1}{3} & > |h(x_{i_n,k_n})-h(x)| \\
& \geq |f_{i_n,k_n}(x_{i_n,k_n})-f_{i_n,k_n}(x)|-|f_{i_n,k_n}(x_{i_n,k_n})-h(x_{i_n,k_n})|- |h(x)-f_{i_n,k_n}(x)| \\
& > (1-1/k_n)-4^{-(i_n+k_n)}-4^{-(i_n+k_n)}>\frac{1}{3},
\end{aligned}
\end{eqnarray*}
which is the desired contradiction.
\end{proof}

We need the following result.
\begin{lemma} \label{l:paracom-pseudo-compact}
Every paracompact locally pseudocompact space $X$ is locally compact.
\end{lemma}
\begin{proof}
Let $x\in X$ and take a  neighborhood $U$ of $x$ with pseudocompact closure $\overline{U}$. Then  $\overline{U}$ is compact being pseudocompact and  paracompact, see, e.g., \cite[3.10.21, 5.1.5, and 5.1.20]{Eng}.
\end{proof}

Now we are ready to prove the main result of this section.

\begin{proof}[Proof of Theorem \ref{t:Ck-Ascoli-point-count-type}]
(i)$\Rightarrow$(ii) follows from \cite[5.1.27]{Eng}.

(ii)$\Rightarrow$(iii),(v): If $X=\bigoplus_{i\in\kappa} X_i$, then
\[
\CC(X) =\prod_{i\in\kappa} \CC(X_i)\quad \mbox{ and } \quad \CC(X,\II) =\prod_{i\in\kappa} \CC(X_i,\II),
\]
where all the spaces $ \CC(X_i)$ and $\CC(X_i,\II)$ are complete metrizable. So $\CC(X)$ and $\CC(X,\II)$ are  $k_\IR$-spaces by \cite[Theorem 5.6]{Nob}.

(iii)$\Rightarrow$(iv) and (v)$\Rightarrow$(vi) follow from \cite{Noble}. The implications (iv)$\Rightarrow$(i) and (vi)$\Rightarrow$(i) follow from Lemmas \ref{l:Ascoli-loc_pseud_comp} and \ref{l:paracom-pseudo-compact}.
\end{proof}

 Theorem \ref{t:Ck-Ascoli-point-count-type} also holds for some
spaces without point-countable type.
\begin{example} \label{exa:Ascoli-non-point-type} {\em
Let $X=D\cup\{ \infty\}$ be the one point Lindel\"{o}fication of an uncountable discrete space $D$. Clearly, $X$ is scattered and Lindel\"{o}f. Since any compact subset of $X$ is finite and $D$ is uncountable, the space $X$ is not of point-countable type. Nevertheless, $\CC(X)=C_p(X)$ is Ascoli by Corollary \ref{c:Cp-Ascoli-finite}.}
\end{example}

The following statement  probably belongs to folklore.

\begin{lemma} \label{l:N_to_omega_1}
Let $X$ be a paracompact space which is not Lindel\"of. Then $\w^{\w_1}$ can be embedded into $\CC(X)$ as a closed subspace, where $\w$ is considered with the discrete topology.
\end{lemma}

\begin{proof}
Since $X$ is paracompact and non-Lindel\"of, Lemma~2.2 of \cite{Bur84} implies that there is an uncountable $A\subset X$ and open $U_a\ni a$ for every $a\in A$ such that each $x\in X$ has a neighbourhood which meets at most one of the $U_a$'s. Set
\[
Z:=\{f\in \CC(X):f\uhr(X\setminus\bigcup_{a\in A}U_a)=0\} \; \mbox{ and }\; Z_a:=\{f\in \CC(X):f\uhr(X\setminus U_a)=0\}.
\]
Then $Z$ is a closed subspace of $\CC(X)$ and $Z=\prod_{a\in A}Z_a$. It suffices to note that each $Z_a$ contains a closed copy of $\IR$ (and hence of $\w$) being a linear topological space.
\end{proof}

Recall that a {\em compact resolution} in a topological space $X$ is a family $\{ K_\alpha:
\alpha\in\NN^\NN\}$ of compact subsets of $X$ which covers $X$ and satisfies the condition: $K_\alpha\subseteq K_\beta$ whenever $\alpha\leq\beta$ for all $\alpha,\beta\in\NN^\NN$.
\begin{lemma} \label{l:comp_res_lind}
Let $X$ be a paracompact space with  compact resolution. Then $X$ is Lindel\"of.
\end{lemma}
\begin{proof}
Suppose for a contradiction that $X$ is not Lindel\"{o}f.  Then $X$ contains a closed discrete uncountable subset $Y$ by  \cite[Lemma~2.2]{Bur84}. Hence the  compact resolution restricted to $Y$ is also a compact resolution on $Y$. So $Y$ is a metric space with a compact resolution. Therefore  $Y$ is separable by \cite[Corollary 6.2]{kak}, and hence it is   countable being discrete, a contradiction.
\end{proof}

Recall that a space $X$ is \emph{hemicompact} if it has a countable family of
compact subspaces which is cofinal with respect to inclusion
in the family of all of its compact subspaces.
The following theorem extends Corollary 4 of \cite{McCoy-1983}.
\begin{theorem} \label{t:Ascoli-compact-res}
Let $X$ be a paracompact space of  point-countable type.  Then the following
conditions are equivalent:
\begin{enumerate}
\item[{\rm (i)}] $X$ is hemicompact;
\item[{\rm (ii)}]  $\CC(X)$ is a $k$-space;
\item[{\rm (iii)}] $\CC(X)$ is Ascoli and $X$ has a compact resolution.
\end{enumerate}
\end{theorem}

\begin{proof}
(i)$\Rightarrow$(ii) is clear. (ii)$\Rightarrow$(iii) Assume  that  $\CC(X)$ is a $k$-space.  Then $\CC(X)$  is
 Ascoli.  Hence  $X$ is locally compact by Theorem \ref{t:Ck-Ascoli-point-count-type}.
Moreover $X$ is Lindel\"{o}f. Indeed, if not, then $\CC(X)$ contains as a closed
 subset the product $\w^{\w_{1}}$ by Lemma \ref{l:N_to_omega_1},  a contradiction since
 $\w^{\w_{1}}$ is not a $k$-space. Hence $X$ is Lindel\"{o}f. Consequently $X$ is
hemicompact. Thus $X$ has a compact resolution.   (iii)$\Rightarrow$(i)
Since $\CC(X)$ is Ascoli,  $X$ is locally compact by
 Theorem \ref{t:Ck-Ascoli-point-count-type}.  Now Lemma \ref{l:comp_res_lind} implies
that $X$ is Lindel\"{o}f, so $X$ is hemicompact.
\end{proof}

We need the following lemma.
\begin{lemma} \label{l:Ck-I-k-space}
Let $X$ be a non-discrete locally compact space. Then $C_p(X,\II)$ contains a closed infinite discrete subspace.
\end{lemma}
\begin{proof}
Let $K$ be an infinite compact subset of $X$. Take a countably infinite discrete
subset $D$ of $K$ and let $x_\ast$ be a limit point of $D$. Set $A:=D\cup\{ x_\ast\}$.
 Then the restriction operator $T:C_p(X,\II)\to C_p(A,\II)$ is continuous, and the
image $E$ of $T$ is dense in the compact metrizable space $\II^A$. As $A$ has a
limit point we obtain $E\not= \II^A$, so we can find $z_\ast \in \II^A\setminus E$.
Let now $B=\{z_n\}\subset E$ be such that $z_n\to z_\ast$. Clearly, $B$ is a discrete
and closed infinite subset of $E$.  For every $b\in B$ select $f_b\in T^{-1}(b)$.
 Then $\{f_b:b\in B\}$ is a desired closed infinite
 discrete subspace of $C_p(X,\II)$.
\end{proof}

The next theorem generalizes a result of R.~Pol \cite{Pol-1974}.
\begin{theorem} \label{t:Ck-k-space-point-count-type}
Let $X$ be a paracompact space of  point-countable type.  Then:
\begin{enumerate}
\item[{\rm (i)}] $\CC(X,\II)$ is a $k$-space if and only if $X$ is the topological
sum of a Lindel\"{o}f locally compact space $L$ and a discrete space $D$;
 so $\CC(X,\II)=\CC(L,\II)\times \II^{|D|}$, where $\CC(L,\II)$ is a
 complete metrizable space;
\item[{\rm (ii)}] $\CC(X,\II)$ is a  sequential space if and only if  $\CC(X,\II)$ is a complete metrizable space if and only if $X$ is a  Lindel\"{o}f locally compact space.
\end{enumerate}
\end{theorem}
\begin{proof}
(i) If $C_k(X,\II)$ is a $k$-space, then $X$ is a locally compact space by Lemmas \ref{l:Ascoli-loc_pseud_comp} and \ref{l:paracom-pseudo-compact}. So $X=\bigoplus_{i\in I} X_i$ is the direct sum of a family $\{ X_i\}_{i\in I}$ of Lindel\"{o}f locally compact spaces by \cite[5.1.27]{Eng}. Denote by $J$ the set of all $i\in I$ for which $X_i$ is not discrete. To prove (i) we have to show that $J$ is countable. Suppose for a contradiction that $J$ is uncountable. Then $C_p(X_i,\II)$ and hence $\CC(X_i,\II)$ contains a closed infinite discrete subspace $D_i$ topologically isomorphic to $\NN$ by Lemma \ref{l:Ck-I-k-space}. So the space
\[
\CC(X,\II)=\prod_{i\in J} \CC(X_i,\II) \times \prod_{i\in I\setminus J} \CC(X_i,\II)
\]
contains $\NN^{|J|}$ as a closed subspace. As $J$ is uncountable we obtain that $\NN^{|J|}$ is not a $k$-space. This contradiction shows that $J$ must be countable. Setting $L:= \bigcup_{i\in J} X_i$ and $D:= \bigcup_{i\in I\setminus J} X_i$ we obtain
the desired decomposition. The converse assertion is trivial.

(ii) If $C(X,\II)$ is a  sequential space, it follows from (i) that $D$ is countable. Indeed, the space $\II^{|D|}$ contains $2^{|D|}$ as a closed subspace and it is well-known that $2^{|D|}$ is sequential (even has countable tightness) if and only if $D$ is countable. So $X$ is a  Lindel\"{o}f locally compact space. If $X$ is Lindel\"{o}f and locally compact space, then $\CC(X)$ and hence its closed subspace $C(X,\II)$ are complete metrizable spaces.
\end{proof}

\bibliographystyle{amsplain}

\begin{thebibliography}{10}

\bibitem{Arhangel}
A. V. Arhangel'skii, \emph{Topological function spaces}, Math. Appl. \textbf{78}, Kluwer Academic Publishers, Dordrecht, 1992.


\bibitem{ArT}
A.~V.~Arhangel'skii,  M.~G.~Tkachenko, \emph{Topological groups and related strutures}, Atlantis Press/World Scientific, Amsterdam-Raris, 2008.


\bibitem{BG}
T. Banakh, S. Gabriyelyan,  On the $\CC$-stable closure of the class of (separable) metrizable spaces, Monatshefte Math.  \textbf{180} (2016), 39--64.

\bibitem{Bor}
C.J.R. Borges, On stratifiable spaces. Pacific J. Math. \textbf{17} (1966), 1--16.

\bibitem{Bur84}
 D.K. Burke,  Covering properties. In: {\em Handbook of set-theoretic topology},  North-Holland, Amsterdam, 1984, 347--422.

\bibitem{cascales} B. Cascales, I. Namioka, \textit{The Lindel\"of property and $\sigma$-fragmentability}, Fund. Math. \textbf{180} (2003), 161--183.


\bibitem{comfort}
W. W. Comfort, Remembering Mel Henriksen and (some of) his theorems, Topology Appl. \textbf{158} (2011), 1742-�1748.



\bibitem{Eng}
R. Engelking, \emph{General topology}, Panstwowe Wydawnictwo Naukowe, Waszawa, 1977.


\bibitem{Gabr-C2}
S. Gabriyelyan, \emph{Topological properties of function spaces $\CC(X,2)$ over zero-dimensional metric spaces $X$},  Topology Appl., accepted.

\bibitem{GKP}
S. Gabriyelyan, J. Kakol, G. Plebanek, The Ascoli property for function spaces and the weak topology on Banach and Fr\'echet spaces, Studia Math.  \textbf{233} (2016), 119--139.

\bibitem{GGKZ-2}
S. Gabriyelyan, J. Greb\'{\i}k, J. K\c{a}kol, L. Zdomskyy, Topological properties of function spaces over ordinal spaces, preprint.


\bibitem{GerNagy}
J. Gerlits, Zs. Nagy,  Some properties of $C(X)$. I, Topology Appl. \textbf{14} (1982), 151--161.

\bibitem{gruenhage}
G. Gruenhage, Generalized metric spaces, \emph{Handbook of Set-theoretic Topology}, North-Holland, New York, 1984, 423--501.

%\bibitem{Hurewicz}
%W. Hurewicz, \"{U}ber Folgen stetiger Funktionen, Fund. Math. \textbf{9} (1927), 193--204.

\bibitem{husek} M. Husek, J. van Mill, \textit{Recent progress in General Topology}, North Holland (1992).


\bibitem{JMSS}
W. Just, A. W. Miller, M. Scheepers, P. J. Szeptycki, The combinatorics of open covers II, Topology Appl. \textbf{73} (1996), 241--266.


\bibitem{kak}
J.~K\c akol, W.~Kubi\'s, M.~Lopez-Pellicer, \emph{Descriptive Topology in Selected  Topics of Functional Analysis}, Developments in Mathematics, Springer, 2011.

\bibitem{LiL}
C.~Liu, L.D.~Ludwig, $\kappa$-Fr\'{e}chet--Urysohn spaces, Houston J. Math. \textbf{31} (2005), 391--401.

\bibitem{McCoy-1983}
R.~A.~McCoy, Complete function spaces,   Internat. J. Math. Math. Sci.  \textbf{6}  (1983),  271--277.

%\bibitem{McCoy-Ntantu}
%R.~A.~McCoy, I.~Ntantu, Completeness of function spaces, Topology Appl. \textbf{22} (1986), 191--206.


\bibitem{mcoy}
R.~A.~McCoy, I.~Ntantu, \emph{Topological Properties of Spaces of Continuous Functions}, Lecture Notes in Math.  \textbf{1315}, 1988.


\bibitem{Mich}
E.~Michael, $\aleph_0$-spaces, J. Math. Mech. \textbf{15} (1966), 983--1002.

\bibitem{Miller}
A. W. Miller, Special subsets of the real line, in: K. Kunen, J.E. Vaughan (Eds.), {\em Handbook of Set-Theoretic Topology}, Elsevier, Amsterdam, 1984, 201--233.

\bibitem{Noble}
N.~Noble, Ascoli theorems and the exponential map, Trans. Amer. Math. Soc. {143} (1969), 393--411.

\bibitem{Nob}
N.~Noble, The continuity of functions on Cartesian products, Trans. Amer. Math. Soc. \textbf{149} (1970), 187--198.

\bibitem{Pol-1974}
R. Pol, Normality in function spaces, Fund. Math. {84} (1974), 145--155.

\bibitem{Rothberger}
F. Rothberger, Sur un ensemble de premiere categorie qui est depourvu de la propriete $\lambda$, Fundamenta Mathematicae \textbf{32} (1939), 50--55.

\bibitem{RudWat83}
M.E. Rudin, S. Watson, Countable products of scattered paracompact spaces,  Proc. Amer. Math. Soc.  \textbf{89}  (1983),  551--552.


\bibitem{Sak2}
M.~Sakai, Two properties of $C_p(X)$ weaker than Fr\'{e}chet-Urysohn property, Topology Appl. \textbf{153} (2006), 2795--2804.

\bibitem{sema}
Z. Semadeni, \emph{Banach spaces of continuous functions},  Monografie Matematyczne {55} PWN-Polish Scientific Publishers, Warszawa, 1971.

\bibitem{Tel68}
R. Telg\'{a}rsky, Total paracompactness and paracompact dispersed spaces, Bull. Acad. Polon. Sci. S\'{e}r. Sci. Math. Astronom. Phys. \textbf{16} (1968), 567--572.

\bibitem{Tkachuk-Book-1}
V. Tkachuk, \emph{A $C_p$-theory problem book, special features of function spaces}, Springer,  New York, 2014.

\end{thebibliography}

\end{document}